\documentclass[12pt,a4paper]{article}
\usepackage[utf8x]{inputenc}
\usepackage{amsmath}
\usepackage{amsfonts}
\usepackage{amssymb}
\usepackage{amsthm}
\usepackage{mathrsfs}
\usepackage{fancyhdr}
\usepackage{graphicx}
\usepackage[usenames,x11names]{xcolor}
\usepackage{arabtex}
\usepackage{framed}
\usepackage[top=2cm,bottom=2cm,margin=2cm]{geometry}
\usepackage{cancel}
\usepackage{array}   %%%% Pour les fonctionalités assez évoluées d'un tableau

\usepackage[dvipdfm,pdfstartview=FitH,colorlinks=true,linkcolor=Red4,urlcolor=Red4,%
filecolor=green,citecolor=red]{hyperref}

\title{Identities and estimations involving the least common multiple of strong divisibility sequences}
\author{\textsc{Sid Ali BOUSLA} and \textsc{Bakir FARHI} \\
Laboratoire de Mathématiques appliquées \\
Faculté des Sciences Exactes \\
Université de Bejaia, 06000 Bejaia, Algeria \\[1mm]
\href{mailto:bouslasidali@gmail.com}{bouslasidali@gmail.com} (S.A. Bousla),  \href{mailto:bakir.farhi@gmail.com}{bakir.farhi@gmail.com} (B. Farhi)
}
\date{}

\let\up=\textsuperscript

\def\N{{\mathbb N}}

\def\restmod#1#2{#1\ (\mathrm{mod}\ #2)}

\def\lcm{\mathrm{lcm}}

\def\EMdash{\leavevmode\hbox to 10.6mm{\vrule height .63ex depth -.59ex
    width 10mm\hfill}}

\theoremstyle{plain}
\numberwithin{equation}{section}
\newtheorem{thm}{Theorem}[section]

\newtheorem{lemma}[thm]{Lemma}

\newtheorem{coll}[thm]{Corollary}

\newtheorem{thm-n}{Theorem} %%%% Théorème numéroté sans se reférer à la section qui le contient.
\newtheorem{prop-n}[thm-n]{Proposition} %%% Proposition numérotée sans se réferer à la section qui la contienne.

\begin{document}
\maketitle

\vspace*{-10cm}

\fbox{\begin{minipage}{0.47\textwidth}
{To appear in \textit{C. R. Acad. Sci. Paris, Sér. I}}
\end{minipage}
}

\vspace*{9cm}

\begin{abstract}
In this paper, we first prove that for any strong divisibility sequences $\boldsymbol{a} = \left(a_n\right)_{n\geq 1}$, we have the identity: $\lcm\left\lbrace \binom{n}{0}_{\bf{a}},\binom{n}{1}_{\bf{a}},\dots,\binom{n}{n}_{\bf{a}}\right\rbrace=\frac{\lcm\left(a_1,\dots,a_n,a_{n+1}\right)}{a_{n+1}}$ $\left(\forall n\geq 1\right)$, generalizing the identity of Farhi (obtained in 2009 for $a_n=n$). Then, we derive from this one some other interesting identities. Finally, we apply those identities to estimate the least common multiple of the consecutive terms of some Lucas sequences. Denoting by $\left(F_n\right)_n$ the usual Fibonacci sequence, we prove for example that for all $n\geq 1$, we have 
\[\Phi^{\frac{n^2}{4}-\frac{9}{4}}\leq\lcm\left(F_1,\dots,F_n\right)\leq\Phi^{\frac{n^2}{3}+\frac{4n}{3}},\]
where $\Phi$ denotes the golden ratio.    
\end{abstract}

\section{Introduction and Notation}
Throughout this paper, we let $\mathbb{N^*}$ denote the set $\mathbb{N}\setminus \left\lbrace 0\right\rbrace$ of positive integers. We denote by $\lfloor .\rfloor$ the integer-part function. If $a_1,a_2,\dots,a_n$ $(n\geq 1)$ are integers not all zero, we let $\gcd\left(a_1,\dots,a_n\right)$ and $\lcm\left(a_1,\dots,a_n\right)$ respectively denote the greatest common positive divisor and the least common positive multiple of $a_1,\dots,a_n$. A sequence of positive integers $\left(a_n\right)_{n\geq 1}$ is simply denoted by $\bf{a}$. Let $\bf{a}$ be a such sequence. For $n\in\mathbb{N}$, we let $[n]_{\bf{a}}!$ denote the positive integer 
\[[n]_{\bf{a}}!:=a_1a_2\cdots a_n\]  
(with the convention $[0]_{\bf{a}}!=1$). For $n,k\in\mathbb{N}$, with $n\geq k$, we let $\binom{n}{k}_{\bf{a}}$ denote the positive rational number 
\[\binom{n}{k}_{\bf{a}}:=\frac{a_na_{n-1}\cdots a_{n-k+1}}{a_1a_2\cdots a_k}=\frac{[n]_{\bf{a}}!}{[k]_{\bf{a}}![n-k]_{\bf{a}}!}.\]
We call those numbers the \emph{$\bf{a}$-binomial coefficients}. Note that the usual binomial coefficients are obtained by taking $a_n = n$ and the gaussian binomial coefficients $\binom{n}{k}_q$ (also called the \linebreak $q$-binomial coefficients) are obtained by taking $a_n = q^n - 1$ (where $q \geq 2$ is an integer). From the definition, we easily check that the $\bf{a}$-binomial coefficients satisfy the following identities:
\begin{equation}\label{1}
\binom{n}{k}_{\bf{a}}=\binom{n}{n-k}_{\bf{a}}~~~~\left(\forall n,k\in\mathbb{N}, n\geq k\right) ,
\end{equation} 
\begin{equation}\label{2}
a_k\binom{n+1}{k}_{\bf{a}}=a_{n+1}\binom{n}{k-1}_{\bf{a}}~~~~\left(\forall n,k\in\mathbb{N}, 1\leq k\leq n+1\right) ,
\end{equation}
\begin{equation}\label{3}
\binom{n}{k}_{\bf{a}}\binom{k}{l}_{\bf{a}}=\binom{n}{l}_{\bf{a}}\binom{n-l}{k-l}_{\bf{a}}~~~~\left(\forall n,k,l\in\mathbb{N}, l\leq k\leq n\right).
\end{equation}
A \emph{strong divisibility sequence} is a sequence of positive integers $\boldsymbol{a}=\left(a_n\right)_{n\geq 1}$ which satisfies for all $n,m\in\mathbb{N^*}$ the property:
\[\gcd\left(a_n,a_m\right)=a_{\gcd(n,m)}.\]
From this definition, it is immediate that any strong divisibility sequence ${(a_n)}_{n \geq 1}$ satisfies the weaker divisibility property:
$$
n \mid m \Longrightarrow a_n \mid a_m ~~~~~~~~~~ (\forall n , m \in \N^*) ,
$$
but the converse is not always true. If the last property is satisfied (but not necessarily the first), we just say that ${(a_n)}_{n \geq 1}$ is a \emph{divisibility sequence}. \\
The sequence of all positive integers $\left(a_n=n, \forall n\geq 1\right)$ is obviously a strong divisibility sequence. An important class of strong divisibility sequences is the Lucas sequences (with some constraints on their parameters). For $P,Q\in\mathbb{Z^*}$, the Lucas sequence $U\left(P,Q\right)$ is the sequence of integers defined by:
\[\begin{cases}U_0=0,~U_1=1\\U_{n+2}=PU_{n+1}-QU_{n}~~\left(\forall n\in\mathbb{N}\right)\end{cases}.\]
If $\Delta:=P^2-4Q>0$, we denote by $\alpha$ and $\beta$ the roots of the quadratic equation: $X^2-PX+Q=0$ such that $\left|\alpha\right|>\left|\beta\right|$. In fact, the sequence $U\left(P,Q\right)$ can be expressed in terms of $\alpha$ and $\beta$. We have the following so called Binet's formula: 
\begin{equation}\label{s}
U_n=\frac{\alpha^n-\beta^n}{\alpha-\beta}~~~~\left(\forall n\in\mathbb{N}\right).
\end{equation}
It is known that if $P$ and $Q$ are coprime then $\left|U\left(P,Q\right)\right|$ is a strong divisibility sequence (see e.g., \cite{div}). Note that those sequences include the sequence of all natural numbers (take $\left(P,Q\right)=\left(2,1\right)$) and also the usual Fibonacci sequence (take $\left(P,Q\right)=\left(1,-1\right)$). For a reading on the Lucas sequences, the reader can consult the book of Honsberger \cite{H}. The general structure of the divisibility sequences was the area of interest of several authors at least since the second half of the 20\up{th} century. In 1936, Ward \cite{Ward} investigated the $p$-adic valuation of such sequences and discovered the property that for any strong divisibility sequence $\boldsymbol{a}$, the $\boldsymbol{a}$-binomial coefficients are all positive integers. In \cite{bez}, Bézivin et al. established a complete characterization of divisibility sequences which are linear recurrent. In \cite{kim2}, Kimberling implicitly obtained an important theorem stating that the general term of a strong divisibility sequence $\boldsymbol{a}=\left(a_n\right)_{n\geq 1}$ can always be made in the form:
\begin{equation}\label{4}
a_{n}=\prod_{d\mid n}u_d~~~~~~~~~~\left(\forall n\in\mathbb{N^*}\right) ,
\end{equation}
for some sequence of positive integers $\left(u_{n}\right)_{n\geq 1}$. From the Möbius inversion formula, we have that 
\[u_n=\prod_{d\mid n}a_d^{\mu\left(\frac{n}{d}\right)}~~~~\left(\forall n\in\mathbb{N^*}\right),\]
where $\mu$ is the well-known Möbius function. More recently, Bliss et al. \cite{Bliss} explicitly rediscovered this theorem and established another important formula of $\left(u_n\right)_{n\geq 1}$ in terms of $\left(a_n\right)_{n\geq 1}$. Using the representation \eqref{4} for the particular case $a_n = q^n - 1$ (where $q \geq 2$ is an integer), Knuth and Wilf showed an important formula for the gaussian binomial coefficients $\binom{n}{k}_q$ (see \cite[Equation (10)]{Knuth}). Actually the Kunth-Wilf formula can be easily generalized to the case of any strong divisibility sequence, as shows the following proposition:
\begin{prop-n}\label{p1}
Let $\left(u_n\right)_{n\geq 1}$ be a sequence of positive integers and $\left(a_n\right)_{n\geq 1}$ be the sequence (of positive integers) whose the general term is given by:
\[a_n=\prod_{d\mid n}u_d~~~~\left(\forall n\in\mathbb{N^*}\right).\]
Then, for all $n,k\in\mathbb{N}$, with $n\geq k$, we have
\[\binom{n}{k}_{\bf{a}}=\prod_{d}u_d,\]
where the product on the right-hand side is taken over all positive integers $d\leq n$ such that:
\[\left\lfloor \frac{k}{d}\right\rfloor +\left\lfloor \frac{n-k}{d}\right\rfloor <\left\lfloor \frac{n}{d}\right\rfloor.\]
In particular, the $\bf{a}$-binomial coefficients $\binom{n}{k}_{\bf{a}}$ $\left(n,k\in\mathbb{N}, n \geq k\right)$ are all positive integers. 
\end{prop-n}
\noindent\textit{Proof.} Let $n , k \in \N$ such that $n \geq k$. We have
$$
\binom{n}{k}_{\bf{a}} := \dfrac{\prod_{n - k < m \leq n} a_m}{\prod_{1 \leq m \leq k} a_m} = \dfrac{\prod_{n - k < m \leq n} \prod_{d \mid m} u_d}{\prod_{1 \leq m \leq k} \prod_{d \mid m} u_d} .
$$
But since
$$
\prod_{n - k < m \leq n} \prod_{d \mid m} u_d = \prod_{1 \leq d \leq n} \prod_{\begin{subarray}{c}
n - k < m \leq n \\
m \equiv \restmod{0}{d}
\end{subarray}} \!\!\! u_d = \prod_{1 \leq d \leq n} u_d^{\lfloor\frac{n}{d}\rfloor - \lfloor\frac{n - k}{d}\rfloor} = \prod_{d \geq 1} u_d^{\lfloor\frac{n}{d}\rfloor - \lfloor\frac{n - k}{d}\rfloor}
$$
and similarly
$$
\prod_{1 \leq m \leq k} \prod_{d \mid m} u_d = \prod_{1 \leq d \leq k} \prod_{\begin{subarray}{c}
1 \leq m \leq k \\
m \equiv \restmod{0}{d}
\end{subarray}} \!\!\! u_d = \prod_{1 \leq d \leq k} u_d^{\lfloor\frac{k}{d}\rfloor} = \prod_{d \geq 1} u_d^{\lfloor\frac{k}{d}\rfloor} ,
$$
then it follows that:
$$
\binom{n}{k}_{\bf{a}} = \prod_{d \geq 1} u_d^{\lfloor\frac{n}{d}\rfloor - \lfloor\frac{k}{d}\rfloor - \lfloor\frac{n - k}{d}\rfloor} .
$$
The required identity of the proposition follows from the fact that:
\begin{equation}
\left\lfloor\frac{n}{d}\right\rfloor - \left\lfloor\frac{k}{d}\right\rfloor - \left\lfloor\frac{n - k}{d}\right\rfloor \in \{0 , 1\} ~~~~~~~~~~ (\forall d \geq 1) . \tag*{$\square$}
\end{equation}

About the converse of the Kimberling representation \eqref{4}, it is known that not any sequence $\bf{a}$ whose general term has the form \eqref{4} is a strong divisibility sequence. In \cite{Bliss}, Bliss et al. successed to establish a necessary and sufficient condition on a sequence $\bf{u}$ so that the sequence $\bf{a}$ defined by \eqref{4} be a strong divisibility sequence. This condition is ``$\forall n,m\in\mathbb{N^*}$ such that $n\nmid m$ and $m\nmid n$: $\gcd\left(u_n,u_m\right)=1$''. More recently, Nowicki \cite{Nowicki} has expanded the above condition of Bliss et al. and obtained a practical necessary and sufficient condition for a sequence of positive integers to be a strong divisibility sequence. This result is the following:
\begin{thm-n}[Nowicki \cite{Nowicki}]\label{p2}
Let $\left(a_n\right)_{n\geq 1}$ be a sequence of positive integers. For all $n\geq 1$, let $c_n$ be the positive integer defined by:
\[c_n:=\frac{\lcm\left(a_1,\dots,a_n\right)}{\lcm\left(a_1,\dots,a_{n-1}\right)}\] 
(with the convention $c_1=a_1$). Then $\left(a_n\right)_n$ is a strong divisibility sequence if and only if we have for all $n\geq 1$:
\[a_n=\prod_{d\mid n}c_d.\]
\end{thm-n}
The arithmetic properties of the binomial coefficients was an old and fascinating subject of study of several authors. For example, Sylvester \cite{syl} proved more than one century ago that if $n$ and $k$ are positive integers such that $n\geq 2k$, then the binomial coefficient $\binom{n}{k}$ contains at least a prime divisor greater than $k$. Recently, Farhi \cite{Farhi} proved the remarkable and interesting identity $\lcm\left\lbrace \binom{n}{0},\binom{n}{1},\dots,\binom{n}{n}\right\rbrace=\frac{\lcm\left(1,2,\dots,n,n+1\right)}{n+1}$ $\left(\forall n\in\mathbb{N}\right)$. This identity was then generalized by Guo \cite{Victor} to $q$-binomial coefficients. In this paper, we obtain a more general identity dealing with the strong divisibility sequence from which Farhi's and Guo's identities become just particular cases (see theorem \ref{R1} and the remarks at the end of the paper). After that, we deduce two other identities that we will use successfully to obtain nontrivial effective estimations for the least common multiple of the consecutive terms of some Lucas sequences (see theorem \ref{R4}). The goodness of our effective estimations is insured by the asymptotic estimations obtained respectively by Matiyasevich and Guy \cite{Mat}, Kiss and Matyas \cite{kiss} for the least common multiple of the same type of sequences.         
\section{The results and the proofs}
Our principal result is the following:
\begin{thm}\label{R1}
Let $\boldsymbol{a}=\left(a_n\right)_{n\geq 1}$ be a strong divisibility sequence. Then, for any non-negative integer $n$, we have:
\begin{equation}\label{rr1}
\lcm\left\lbrace \binom{n}{0}_{\bf{a}},\binom{n}{1}_{\bf{a}},\dots,\binom{n}{n}_{\bf{a}}\right\rbrace=\frac{\lcm\left(a_1,a_2,\dots,a_n,a_{n+1}\right)}{a_{n+1}}.
\end{equation}
\end{thm}
To prove this theorem, we need the following lemma of Guo \cite{Victor}:
\begin{lemma}[Guo \cite{Victor}]\label{g}
Let $n$ and $d$ be two positive integers with $n\geq d$. Then, the two following properties are equivalent:
\begin{enumerate}
\item There exists $k\in \left\lbrace 0,1,\dots,n\right\rbrace$ such that: $\left\lfloor \frac{k}{d} \right\rfloor +\left\lfloor \frac{n-k}{d} \right\rfloor < \left\lfloor \frac{n}{d}\right\rfloor$. 
\item The positive integer $d$ does not divide $(n+1)$.
\end{enumerate}
\end{lemma}
\begin{proof}[Proof of Theorem \ref{R1}]
Let $n\in\mathbb{N}$ be fixed. For $n=0$, the identity of the theorem is trivial. Suppose for the sequel that $n\geq 1$. Let $A_n$ and $B_n$ respectively denote the left-hand side and the right-hand side of \eqref{rr1}. So, we have to show that $A_n=B_n$. To do so, we first show that $A_n$ divides $B_n$ and then that $B_n$ divides $A_n$. Since $\bf{a}$ is a strong divisibility sequence then, according to Theorem \ref{p2}, we have for any $m\geq 1$:
\[a_m=\prod_{d\mid m}u_d,\]
where $\left(u_d\right)_{d\geq 1}$ is the sequence of positive integers defined by:
\[u_1:=a_1~~\text{and}~~u_d:=\frac{\lcm\left(a_1,\dots,a_d\right)}{\lcm\left(a_1,\dots,a_{d-1}\right)}~~~~\left(\forall d\geq 2\right).\]
From this definition of $\left(u_d\right)_{d}$, it is immediate that:
\[\lcm\left(a_1,\dots,a_{m}\right)=\prod_{d=1}^{m}u_d~~~~\left(\forall m\geq 1\right).\]
Now, for any $k\in\left\lbrace 0,1,\dots,n\right\rbrace$, the $u_d$'s appearing in the product $\binom{n}{k}_{\bf{a}}=\prod_{d}u_d$ of Proposition \ref{p1} correspond (according to Lemma \ref{g}) to indices $d$ such that $1\leq d\leq n$ and $d\nmid (n+1)$. This implies that the product:
\[\prod_{\begin{subarray}{c}1\leq d\leq n \\ d\nmid (n+1) \end{subarray}}u_d=\frac{\displaystyle\prod_{1\leq d\leq n+1}u_d}{\displaystyle\prod_{d\mid (n+1)}u_d}=\frac{\lcm\left(a_1,\dots,a_{n},a_{n+1}\right)}{a_{n+1}}=B_n\]
is a multiple of each $\binom{n}{k}_{\bf{a}}$ $\left(0\leq k\leq n\right)$. Thus $B_n$ is a multiple of $\lcm\left\lbrace \binom{n}{0}_{\bf{a}},\binom{n}{1}_{\bf{a}},\dots,\binom{n}{n}_{\bf{a}}\right\rbrace=A_n$; that is $A_n\mid B_n$.\\
Next, it is immediate that $\lcm\left(a_{1},\dots,a_{n+1}\right)$ divides $\lcm\left\lbrace a_1\binom{n+1}{1}_{\bf{a}},a_2\binom{n+1}{2}_{\bf{a}},\dots,a_{n+1}\binom{n+1}{n+1}_{\bf{a}}\right\rbrace$, which is (according to \eqref{2}) equal to
\[\lcm\left\lbrace a_{n+1}\binom{n}{0}_{\bf{a}},a_{n+1}\binom{n}{1}_{\bf{a}},\dots,a_{n+1}\binom{n}{n}_{\bf{a}}\right\rbrace=a_{n+1}\lcm\left\lbrace\binom{n}{0}_{\bf{a}},\binom{n}{1}_{\bf{a}},\dots,\binom{n}{n}_{\bf{a}}\right\rbrace=a_{n+1}A_{n}.\]
Hence $\frac{\lcm\left(a_{1},\dots,a_n,a_{n+1}\right)}{a_{n+1}}=B_n$ divides $A_n$.
This completes the proof.
\end{proof}
\noindent\textbf{Remark:} The particular case of the identity of Theorem \ref{R1} corresponding to the sequence of all positive integers (that is $a_n=n$, $\forall n\geq 1$) is the main result of the paper \cite{Farhi} of Farhi.

From Theorem \ref{R1}, we derive two important corollaries:
\begin{coll}\label{R2}
Let $\boldsymbol{a}=\left(a_n\right)_{n\geq 1}$ be a strong divisibility sequence. Then, for any positive integer $n$, we have:
\[\lcm\left(a_1,a_2,\dots,a_n\right)=\lcm\left\lbrace a_1\binom{n}{1}_{\bf{a}},\dots,a_n\binom{n}{n}_{\bf{a}}\right\rbrace.\] 
\end{coll}
\begin{proof}
For any positive integer $n$, we have according to Formula \eqref{2}:
\begin{align*}
\lcm\left\lbrace a_1\binom{n}{1}_{\bf{a}},\dots,a_{n}\binom{n}{n}_{\bf{a}}\right\rbrace &=\lcm\left\lbrace a_{n}\binom{n-1}{0}_{\bf{a}},a_{n}\binom{n-1}{1}_{\bf{a}},\dots,a_{n}\binom{n-1}{n-1}_{\bf{a}}\right\rbrace\\ &=a_{n}\lcm\left\lbrace \binom{n-1}{0}_{\bf{a}},\binom{n-1}{1}_{\bf{a}},\dots,\binom{n-1}{n-1}_{\bf{a}}\right\rbrace \\&=\lcm\left(a_{1},\dots,a_{n}\right)~~~~\text{(by Theorem \ref{R1}),}
\end{align*}
as required.
\end{proof}
\noindent\textbf{Remark:} The particular case of the identity of Corollary \ref{R2} corresponding to the sequence of all positive integers (that is $a_n=n$, $\forall n\geq 1$) is already obtained by Nair in \cite{Nair}.
\begin{coll}\label{R3}
Let $\boldsymbol{a}=\left(a_n\right)_{n\geq 1}$ be a strong divisibility sequence. Then, for any positive integer $n$, we have:
\[\lcm\left(a_1,a_2,\dots,a_n\right)=\gcd\left\lbrace \binom{n}{k}_{\bf{a}}\lcm\left(a_1,\dots,a_k\right);~n/2\leq k\leq n \right\rbrace.\]
\end{coll} 
To give a more proper proof of Corollary \ref{R3}, we shall first prove the following elementary lemma:
\begin{lemma}\label{ccc1}
Let $n$ and $m$ be two positive integers. Let also $a_{1},\dots,a_{n}$, $b_{1},\dots,b_{m}$ be positive integers. Then the property saying that $a_i$ divides $b_j$ for all $1\leq i\leq n$, $1\leq j\leq m$ is equivalent to the property saying that $\lcm\left(a_1,\dots,a_n\right)$ divides $\gcd\left(b_1,\dots,b_m\right)$.
\end{lemma}
\begin{proof}
The property ``$a_i\mid b_j$, $\forall i=1,\dots,n$, $\forall j=1,\dots,m$" is equivalent to say that \linebreak ``$a_{i}\mid \gcd\left(b_{1},\dots,b_{m}\right)$, $\forall i=1,\dots,n$"; that is ``$\gcd\left(b_{1},\dots,b_{m}\right)$ is a multiple of each $a_i$ \linebreak $\left(1\leq i\leq n\right)$", which is equivalent to say that ``$\gcd\left(b_{1},\dots,b_{m}\right)$ is a multiple of $\lcm\left(a_{1},\dots,a_{n}\right)$", as required.
\end{proof}

\begin{proof}[Proof of Corollary \ref{R3}]
Let $n$ be a fixed positive integer. For $k,l \in \N$ such that $n/2\leq k\leq n$ and $1\leq l\leq k$, the positive integer $a_{l}\binom{n}{l}_{\bf{a}}$ obviously divides the positive integer $a_{l}\binom{n}{l}_{\bf{a}}\binom{n-l}{k-l}_{\bf{a}}$, which is (according to Formula \eqref{3}) equal to $a_{l}\binom{k}{l}_{\bf{a}}\binom{n}{k}_{\bf{a}}$. Next, the latter positive integer $a_{l}\binom{k}{l}_{\bf{a}}\binom{n}{k}_{\bf{a}}$ obviously divides the positive integer $\binom{n}{k}_{\bf{a}}\lcm\left\lbrace a_{i}\binom{k}{i}_{\bf{a}};~i=1,\dots,k\right\rbrace$, which is (according to Corollary \ref{R2}) equal to $\binom{n}{k}_{\bf{a}}\lcm\left(a_{1},\dots,a_{k}\right)$. Consequently, for all $k , l \in \N$ such that $n/2 \leq k \leq n$ and $1\leq l\leq k$, we have:
\begin{equation}\label{ccc2}
a_{l}\binom{n}{l}_{\bf{a}}~\text{divides}~\binom{n}{k}_{\bf{a}}\lcm\left(a_{1},\dots,a_{k}\right) .
\end{equation}
We state that \eqref{ccc2} holds even if $n/2\leq k\leq n$ and $k<l\leq n$. Indeed, if $k,l\in \N$ such that $n/2\leq k\leq n$ and $k+1\leq l\leq n$, we have that $1\leq n-l+1\leq n-k\leq k$ and $a_{n-l+1}\binom{n}{n-l+1}_{\bf{a}}=a_{l}\binom{n}{l}_{\bf{a}}$. So, the application of \eqref{ccc2} for $l'=n-l+1$ instead of $l$ confirms the announced statement. Now, by applying Lemma \ref{ccc1} for all the divisibility relations given by \eqref{ccc2} when $1\leq l\leq n$ and $n/2\leq k\leq n$, we derive that
\[\lcm\left\lbrace a_{l}\binom{n}{l}_{\bf{a}};~l=1,\dots,n\right\rbrace~\text{divides}~\gcd\left\lbrace \binom{n}{k}_{\bf{a}}\lcm\left(a_{1},\dots,a_{k}\right);~n/2\leq k\leq n\right\rbrace;\]
that is (according to Corollary \ref{R2}):
\[\lcm\left(a_{1},\dots,a_{n}\right)~\text{divides}~\gcd\left\lbrace\binom{n}{k}_{\bf{a}}\lcm\left(a_{1},\dots,a_{k}\right);~n/2\leq k\leq n\right\rbrace.\] 
The identity of Corollary \ref{R3} follows by observing that
\[\lcm\left(a_{1},\dots,a_{n}\right)=\binom{n}{n}_{\bf{a}}\lcm\left(a_{1},\dots,a_{n}\right)\in\left\lbrace\binom{n}{k}_{\bf{a}}\lcm\left(a_{1},\dots,a_{k}\right);~n/2\leq k\leq n\right\rbrace.\]
The proof of Corollary \ref{R3} is complete.
\end{proof}

Now, from Corollaries \ref{R2} and \ref{R3}, we derive significative and nontrivial effective estimations for the least common multiple of the first consecutive terms of some type of Lucas sequences. We have the following:
\begin{thm}\label{R4}
Let $P$ and $Q$ be two coprime non-zero integers such that $\Delta:=P^2-4Q>0$ and let $U\left(P,Q\right)$ be the associated Lucas sequence. Then, for any positive integer $n$, we have:
\begin{equation}\label{sh}
\left|\alpha\right|^{\frac{n^2}{4}-\frac{n}{2} - 1} \leq \lcm\left(U_1,U_2,\dots,U_n\right) \leq \left|\alpha\right|^{\frac{n^2}{3}+\frac{7n}{3}-\frac{8}{3}},
\end{equation}
where $\alpha$ is the largest root in absolute value of the quadratic equation $X^2-PX+Q=0$.
\end{thm}
To prove this theorem, we need the following lemma:
\begin{lemma}\label{ter}
In the same situation with Theorem \ref{R4}, we have for any positive integer $n$:
\[\left|\alpha\right|^{n-2}\leq\left|U_n\right|\leq\left|\alpha\right|^{n}.\] 
\end{lemma}
\begin{proof}
Let $\beta$ denote the second root of the quadratic equation $X^2-PX+Q=0$; so $|\beta|<|\alpha|$. We have that $|\alpha|-|\beta|\in\left\lbrace \alpha-\beta,\beta-\alpha,\alpha+\beta,-\alpha-\beta\right\rbrace$. But since $\alpha-\beta=\pm\sqrt{\Delta}$, $\alpha+\beta=P$ and $|\alpha|-|\beta|>0$, it follows that $|\alpha|-|\beta|\in\left\lbrace \left|P\right|,\sqrt{\Delta}\right\rbrace$. Then, since $P\in\mathbb{Z}^*$ and $\Delta\in\mathbb{Z}^{*}_{+}$, we deduce that:
\begin{equation}\label{supa1}
|\alpha|-|\beta|\geq 1.
\end{equation}
Using Formula \eqref{s} and \eqref{supa1}, we have for any positive integer $n$:
\begin{align*}
\left|U_n\right|&=\left|\frac{\alpha^n-\beta^n}{\alpha-\beta}\right|=\left|\beta^{n-1}\sum_{k=0}^{n-1}\left(\frac{\alpha}{\beta}\right)^{k}\right| \\ &\leq\left|\beta\right|^{n-1}\sum_{k=0}^{n-1}\left|\frac{\alpha}{\beta}\right|^{k}\\&=\frac{|\alpha|^n-|\beta|^n}{|\alpha|-|\beta|}\\&\leq|\alpha|^n-|\beta|^n\leq |\alpha|^n.
\end{align*}
On the other hand, we have for any integer $n\geq 2$:
\begin{align*}
\left|U_n\right| &= \left|\frac{\alpha^n-\beta^n}{\alpha-\beta}\right| = \frac{\left|\alpha^n-\beta^n\right|}{\left|\alpha-\beta\right|} \geq \frac{\left|\left|\alpha^n\right|-\left|\beta^n\right|\right|}{\left|\alpha-\beta\right|} = \frac{\left|\alpha\right|^n-\left|\beta\right|^n}{\left|\alpha-\beta\right|} \\
&= \frac{\left(|\alpha|-|\beta|\right)\left(|\alpha|^{n-1}+|\alpha|^{n-2}\cdot|\beta|+\dots+|\alpha|\cdot|\beta|^{n-2}+|\beta|^{n-1}\right)}{\left|\alpha-\beta\right|}\\&\geq \frac{\left(|\alpha|-|\beta|\right)\left(|\alpha|^{n-1}+|\alpha|^{n-2}\cdot|\beta|\right)}{\left|\alpha-\beta\right|}\\&=\left|\alpha+\beta\right|\cdot |\alpha|^{n-2}\\&=\left|P\right|\cdot |\alpha|^{n-2}\geq |\alpha|^{n-2}.
\end{align*}
By remarking that $\left|U_n\right|\geq |\alpha|^{n-2}$ is also valid for $n=1$ (since $U_1=1$ and $|\alpha|\geq |\alpha|-|\beta|\geq 1$), we conclude that for any positive integer $n$, we have:
\[|\alpha|^{n-2}\leq \left|U_n\right|\leq |\alpha|^{n},\]
as required. The lemma is proved.
\end{proof}

\begin{proof}[Proof of Theorem \ref{R4}]
Let $\beta$ denote the second root of the quadratic equation $X^2-PX+Q=0$; so $|\beta|<|\alpha|$. By applying the estimation of Lemma \ref{ter} for $n\geq 2$ and just replace $U_1$ by $1$, we immediately deduce that for all positive integers $n$ and $k$ such that $n\geq k$, we have: 
\begin{equation}\label{R42}
\left|\alpha\right|^{k\left(n-k-2\right)+1}\leq \left|\binom{n}{k}_{\bf{U}}\right|\leq \left|\alpha\right|^{k(n-k+2)-1}.
\end{equation}
First, let us show the left inequality of \eqref{sh}. For $n=1$, this inequality is trivial. Next, by using successively Corollary \ref{R2}, Lemma \ref{ter} and then \eqref{R42}, we have for any integer $n\geq 2$:    
\begin{align*}
\lcm\left(U_1,U_{2},\dots,U_{n}\right) &= \lcm\left\lbrace U_{1}\binom{n}{1}_{\bf{U}},U_{2}\binom{n}{2}_{\bf{U}},\dots,U_{n}\binom{n}{n}_{\bf{U}}\right\rbrace \\
&\geq \max_{1\leq k\leq n}\left\lbrace \left|U_{k}\right|\cdot\binom{n}{k}_{\bf{U}}\right\rbrace \\
&\geq \max_{1\leq k\leq n}|\alpha|^{k(n-k-1)-1}\\&\geq |\alpha|^{\left\lfloor \frac{n}{2}\right\rfloor\left(n-\left\lfloor \frac{n}{2}\right\rfloor-1\right)-1} \\
&= \left|\alpha\right|^{n^2/4 - n/2 - 1 + (n/2 - \lfloor n/2\rfloor) - (n/2 - \lfloor n /2\rfloor)^2} \\
&\geq \left|\alpha\right|^{n^2/4 -n/2 - 1} ~~~~~~~~~~ \text{(since $n/2 - \lfloor n/2\rfloor \in [0 , 1[$)},
\end{align*}
as required. The left inequality of \eqref{sh} is proved. Now, let us prove the right inequality of \eqref{sh}; that is $\lcm\left(U_1,U_{2},\dots,U_{n}\right)\leq \left|\alpha\right|^{\frac{n^2}{3}+\frac{7n}{3}-\frac{8}{3}}$ $(\forall n\geq 1)$. To do so, we argue by induction on $n$. For $n \in \{1 , 2 , 3\}$, we have:
\begin{align*}
\lcm(U_1 , U_2 , \dots , U_n) &= \lcm(U_2 , \dots , U_n) ~~~~~~~~~~ (\text{since } U_1 = 1) \\
&\leq \left|U_2 U_3 \cdots U_n\right| \\
&\leq {|\alpha|}^{2 + 3 + \dots + n} ~~~~~~~~~~ (\text{according to Lemma \ref{ter}}) \\
&= {|\alpha|}^{\frac{n^2 + n - 2}{2}} ,
\end{align*} 
which is stronger than what it is required. For $m \geq 2$, suppose that the right inequality of \eqref{sh} holds for any positive integer $n<2m$ and let us show that it also holds for $n = 2m$ and for $n = 2m+1$. By using successively Corollary \ref{R3}, the induction hypothesis and \eqref{R42}, we have:
\begin{align*}
\lcm\left(U_1,U_{2},\dots,U_{2m}\right)&\leq \lcm\left(U_1,U_{2},\dots,U_{m}\right)\cdot\left|\binom{2m}{m}_{\bf{U}}\right|\\&\leq |\alpha|^{\frac{m^2}{3}+\frac{7m}{3}-\frac{8}{3}}\cdot |\alpha|^{m^2+2m-1}\\&=|\alpha|^{\frac{4m^2}{3}+\frac{13m}{3}-\frac{11}{3}}\\&\leq |\alpha|^{\frac{\left(2m\right)^2}{3}+\frac{7\left(2m\right)}{3}-\frac{8}{3}},      
\end{align*}
as required. Similarly, we have:
\begin{align*}
\lcm\left(U_1,U_{2},\dots,U_{2m+1}\right)&\leq \lcm\left(U_1,U_{2},\dots,U_{m+1}\right)\cdot\left|\binom{2m+1}{m+1}_{\bf{U}}\right|\\&=\lcm\left(U_1,U_{2},\dots,U_{m+1}\right)\cdot\left|\binom{2m+1}{m}_{\bf{U}}\right|\\&\leq |\alpha|^{\frac{(m+1)^2}{3}+\frac{7(m+1)}{3}-\frac{8}{3}}\cdot |\alpha|^{m^2+3m-1}\\&=|\alpha|^{\frac{4m^2}{3}+6m-1}\\&\leq |\alpha|^{\frac{4m^2}{3}+6m}=|\alpha|^{\frac{(2m+1)^2}{3}+\frac{7(2m+1)}{3}-\frac{8}{3}}, 
\end{align*}
as required. This achieves this induction and confirms that the right inequality of \eqref{sh} is valid for any $n\geq 1$. The proof of the theorem is complete.
\end{proof}

\noindent\textbf{Remarks:}
\begin{enumerate}
\item In the situation of Theorem \ref{R4}, if the integers $P$ and $Q$ have some particular signs then the double inequality of Lemma \ref{ter} can be slightly improved. For example, if $Q>0$, we have that:
\[|\alpha|^{n-1}\leq\left|U_n\right|\leq |\alpha|^{n}~~~~(\forall n\geq 1).\] 
Also, if $P>0$ and $Q<0$, we have that:
\[|\alpha|^{n-2}\leq\left|U_n\right|\leq |\alpha|^{n-1}~~~~(\forall n\geq 1).\]
So, in these cases, by repeating the proof of Theorem \ref{R4} and using those new inequalities (instead of those of Lemma \ref{ter}), we slightly improve the result of Theorem \ref{R4}. Doing so for the usual Fibonacci sequence (which corresponds to $P=1>0$ and $Q=-1<0$), we obtain that for any positive integer $n$, we have:
\begin{equation}\label{sep}
\Phi^{\frac{n^2}{4}-\frac{9}{4}}\leq \lcm\left(F_1,F_2,\dots,F_n\right)\leq \Phi^{\frac{n^2}{3}+\frac{4n}{3}},
\end{equation}   
where $\Phi$ denotes the golden ratio ($\Phi:=\frac{1+\sqrt{5}}{2}$). The goodness of Estimation \eqref{sep} can be appreciated from the famous result of Matiyasevich and Guy \cite{Mat} stating that:
\[\lim_{n\longrightarrow +\infty}\frac{\log \lcm\left(F_1,F_2,\dots,F_n\right)}{n^2\log \Phi}=\frac{3}{\pi^2}.\]
This last result implies that if $\lambda_1,\mu_1,\eta_1,\lambda_2,\mu_2,\eta_2\in\mathbb{R}$ satisfy:
\[\Phi^{\lambda_1n^2+\mu_1n+\eta_1}\leq\lcm\left(F_1,F_2,\dots,F_n\right)\leq \Phi^{\lambda_2n^2+\mu_2n+\eta_2}~~~~\left(\forall n\geq 1\right)\] 
then we have necessary $\lambda_1\leq \frac{3}{\pi^2}$ and $\lambda_2\geq \frac{3}{\pi^2}$. Since \eqref{sep} corresponds to $\lambda_1=\frac{1}{4}=0.25$ and $\lambda_2=\frac{1}{3}=0.33\dots$ and since $\frac{3}{\pi^2}=0.303\dots$, we see that \eqref{sep} is good enough.
\item The results of Theorem \ref{R1} and corollaries \ref{R2} and \ref{R3} can be easily generalized to any other unique factorization domain $A$ instead of $\mathbb{Z}$ (we refer the reader to the article of Bliss et al. \cite{Bliss} for the general definition and properties of strong divisibility sequences in a unique factorization domain). If we take for example $A=\mathbb{Z}[q]$ and $\boldsymbol{a}=\left(a_n\right)_{n\geq 1}$ the sequence of polynomials of $\mathbb{Z}[q]$ defined by $a_n=[n]_{q}:=\frac{q^n-1}{q-1}$, we obtain Guo's identity \cite{Victor}:
\[\lcm\left\lbrace\binom{n}{0}_{q},\binom{n}{1}_{q},\dots,\binom{n}{n}_{q}\right\rbrace=\frac{\lcm\left([1]_q , [2]_q , \dots , [n+1]_q\right)}{[n+1]_{q}}~~~~\left(\forall n\geq 1\right),\]
where $[k]_{q}$ and $\binom{n}{k}_{q}$ $(0\leq k\leq n)$ are the standard notations in $q$-calculus; that is $[k]_{q} := \frac{q^k - 1}{q - 1}$ and $\binom{n}{k}_{q} := \frac{[n]_q [n - 1]_q \cdots [n - k + 1]_q}{[1]_q [2]_q \cdots [k]_q}$.
\item An alternative proof of Theorem \ref{R1} can be provided by investigating the $p$-adic valuation of the $\bf{a}$-binomial coefficients and using the generalized Legendre's formula given by Ward \cite{Ward}.  
\end{enumerate}

\end{document}